\documentclass[10pt,reqno,twoside,paper=letter]{amsart}
\usepackage{calc}
\usepackage[letterpaper,text={16.7cm,21.5cm},right=14mm,layoutvoffset=10mm]{geometry}
\usepackage{abstract}
\usepackage[dvips]{graphicx}         
\usepackage{paralist}
\usepackage{amsthm}
\theoremstyle{remark}  

\newenvironment{my.proof}[1][Proof.]{\begin{proof}[\rm
\textbf{\phantom{x}#1}]}{\end{proof}} 

\usepackage{float}
\usepackage{enumerate}
\usepackage[symbol*]{footmisc}
\usepackage{ragged2e}
\usepackage{changepage}
\usepackage{layouts}

\usepackage[bf,small]{titlesec} 
\titlelabel{\thetitle.\,\,\,} 
\usepackage[comma,longnamesfirst,sectionbib]{natbib}
\theoremstyle{plain}
\newtheorem{theorem}{Theorem}[section]
\newtheorem{lemma}[theorem]{Lemma}
\newtheorem{proposition}[theorem]{Proposition}
\newtheorem{corollary}[theorem]{Corollary}

\theoremstyle{definition}
\newtheorem{definition}[theorem]{Definition}

\usepackage{amssymb} 
\usepackage{amsfonts}
\usepackage{multicol}
\usepackage{fancyhdr}
\usepackage{cancel}
\usepackage{pdfsync}
\usepackage{hyperref}
\usepackage{calc}
\usepackage{enumerate,amssymb}
\usepackage{latexsym, amscd}
\usepackage{color}
\usepackage[all]{xy}
\usepackage{mathrsfs}
\newcommand{\UH}{{U(\mathcal{H})}}
\newcommand{\PUH}{{PU(\mathcal{H})}}
\newcommand{\Fred}{{\mathrm{Fred}(\mathcal H)}}
\newcommand{\BH}{{\mathcal{B}(\mathcal{H})}}
\newcommand{\HH}{\mathcal{H}}
\newcommand{\CC}{\mathcal{C}}
\newcommand{\SSS}{\mathcal{S}}

 \newcommand{\map}{\mathrm{map}}
\newcommand{\Ext}{\mathrm{Ext}}

\newcommand{\homst}[2]{\mathrm{hom}_{\rm{st}}(#1,#2)}
\newcommand{\sg}{\mathrm{Gr}(\HH)}

\newcommand{\Gr}{\mathrm{Gr}}

\newcommand{\im}{\mathrm{Im}\, }
\newcommand{\Id}{\mathrm{Id}}

\usepackage{mathalfa}
\newcommand{\D}{\mathscr{D}}
\newcommand{\Irrep}{\mathrm{Irrep}}

\def\too{\longrightarrow}
\def\mapstoo{\longmapsto}

\usepackage[sc]{mathpazo}
\usepackage[T1]{fontenc}
\newcommand{\inforevista}{\scriptsize  Rev.\ Acad.\
Colomb.\ Cienc.\ xx(xx) xxx  2016}
\begin{document}
%
%
\pagenumbering{arabic}
\fancypagestyle{plain}{%
\fancyhf{} 
\fancyfoot[R]{\thepage} %
\fancyhead[L]{\inforevista}
\renewcommand{\headrulewidth}{0pt}
\renewcommand{\footrulewidth}{0pt}}
\thispagestyle{plain} 
\pagestyle{fancy}     
\fancyhead{} 
\renewcommand{\headrulewidth}{0.0pt} 
\fancyhead[LO]{\inforevista}
\fancyhead[RO]{\scriptsize Topological properties of spaces of projective unitary representations}
\fancyhead[LE]{\scriptsize Jes\'us  Espinoza \& Bernardo Uribe}
\fancyhead[RE]{\inforevista}
\fancyfoot{} 
\fancyfoot[LE]{\thepage}
\fancyfoot[RO]{\thepage}
\begin{flushright}
\rule[-0.5ex]{0.5ex}{3.0ex} {\large Ciencias exactas}
\end{flushright}
\vspace*{0.3cm}

\begin{center}
{\LARGE \textbf{Topological properties of spaces of projective unitary representations \\}}
\end{center}

\vspace{3mm}
\begin{center}
\textbf{\small Jes\'us Espinoza${}^{1},$ and Bernardo Uribe${}^{2,}\footnotemark[1]$}
\vspace{3mm}

{\scriptsize
${}^{1}$Licenciatura en Matem\'aticas Aplicadas, Universidad  del Papaloapan, Av. Ferrocarril s/n. 68400. Ciudad Universitaria. Campus Loma Bonita, Oaxaca, M\'exico, jespinoza@unpa.edu.mx\\
${}^{2}$Departamento de Matem\'aticas y Estad\'istica, Universidad del Norte, Km 5 V\'ia Puerto Colombia, Barranquilla, Colombia,
bjongbloed@uninorte.edu.co\\
}
\end{center}

\begin{Small}
\vspace{3.0mm}
\rule{\textwidth}{0.4pt}

\begin{center}
\begin{minipage}{14cm}
\vspace{3mm}
\textbf{Abstract}
\vspace{3mm}

{Let $G$ be a compact and connected Lie group and $\PUH$ be the
group of projective unitary operators on an infinite dimensional separable Hilbert space $\HH$ endowed
with the strong operator topology. We study the space $\homst{G}{\PUH}$
of continuous homomorphisms from $G$ to $\PUH$ which are stable, namely  the
homomorphisms whose induced representation contains each irreducible
representation an infinitely number of times. We show that the connected 
components of $\homst{G}{\PUH}$ are parametrized by the isomorphism classes
of $S^1$-central extensions of $G$, and that each connected component
has the group $\hom(G,S^1)$ for fundamental group and trivial higher homotopy groups.
We study the conjugation map $\PUH \to \homst{G}{\PUH}$,
$F \mapsto F\alpha F^{-1}$, we show that it has no local cross sections 
and we prove that for a map $B \to \homst{G}{\PUH}$ with $B$ paracompact of finite covering dimension,  local
lifts to $\PUH$ do exist.
\\[1mm]

\textbf{Key words:} Unitary Representation, Projective Unitary Representation.}
\vspace{3mm}

\textbf{Propiedades topol\'ogicas del espacio de representaciones
unitarias proyectivas.}
\vspace{3mm}

\textbf{Resumen}
\vspace{3mm}

Sea $G$ un grupo de Lie compacto y conexo y $\PUH$ el grupo de operadores
proyectivos e unitarios en un espacio de Hilbert separable e infinito dimensional $\HH$,
 provisto de la topolog\'ia fuerte de operadores. Estudiamos el espacio $\homst{G}{\PUH}$
de homomorfismos continuos desde $G$ a $\PUH$ que son estables, es decir homomorfismos
cuyas representaciones inducidas contienen cada representaci\'on irreducible un n\'umero infinito de veces.
Demostramos que las componentes conexas del espacio $\homst{G}{\PUH}$ est\'an parametrizadas
por las clases de isomorf\'ia de extensiones centrales de $G$ por el grupo $S^1$, y que cada
componente conexa tiene por grupo fundamental al grupo $\hom(G,S^1)$ y sus grupos de homotop\'ia superiores son triviales.
Estudiamos la aplicaci\'on conjugaci\'on $\PUH \to \homst{G}{\PUH}$,
$F \mapsto F\alpha F^{-1}$, demostramos que no tiene secciones locales 
y demostramos que para cualquier aplicaci\'on continua $B \to \homst{G}{\PUH}$ con $B$ paracompacto con dimensi\'on
de cubierta finita (dimensi\'on de Lebesgue), los levantamientos locales a $\PUH$ 
s\'i existen.
\\[1mm]

\textbf{Palabras clave:} Representaci\'on Unitaria, Representaci\'on 
Proyectiva Unitaria.\\
\end{minipage}
\end{center}

\rule{\textwidth}{0.4pt}
\end{Small}

\begin{small}
\columnsep 0.5 cm
\begin{multicols}{2}

\setlength{\parskip}{.3cm}

\section{\small Introduction}
The motivation to study the topological properties of the space $\homst{G}{\PUH}$
of stable homomorphisms from a compact Lie group $G$ to
the group of projective unitary operators on a Hilbert space $\HH$
endowed with the topology of pointwise convergence, comes
from realm of equivariant K-theory.

By a theorem of Atiyah and J{\"a}nich \citet{Janich} the $K$-theory groups of a topological space $X$
may be obtained as the homotopy groups of the space $$\map(X, \Fred)$$ of continuous maps from $X$ to 
the space $\Fred$ of Fredholm operators on $\HH$. Given any projective unitary bundle over
$X$, namely a $\PUH$-principal bundle $\PUH \to P \to X$, we may define a twisted version
of the  K-theory groups by taking the homotopy groups of the space of sections
of the associated $\Fred$ bundle $$P \times_\PUH \Fred \to X.$$ These groups are
called the twisted K-theory groups and they define a parametrized cohomology theory
in the sense of \citet{MaySigurdsson} whenever we consider the category 
of pairs $(X,f)$ with $X$ a topological space and $f:X \to B\PUH$ a map which recovers 
the $\PUH$ bundle $P$.

In the equivariant setup, namely when we consider the category of spaces with $G$ actions, 
the definition of the twisted equivariant K-theory is more intricate. We need to consider
$G$ equivariant projective unitary stable bundles, namely $G$ equivariant $\PUH$-principal
bundles $P \to X$, such that the induced local homomorphism $G_x \to \PUH$ is stable
for $G_x$ the isotropy group of any $x \in X$, in order to define the twisted equivariant
K-theory groups as the homotopy groups of the $G$ invariant sections of the associated
bundle $$P \times_\PUH \Fred \to X.$$ To prove that the twisted equivariant K-theory
is a parametrized cohomology theory in the sense of \citet{May:EHCT} we would
need to construct a  universal $G$ equivariant projective
unitary stable bundle as it was done in the non-equivariant case.
The construction of this universal space can be done using classifying spaces of families
of subgroups as it was done in \citet{LueckUribe}, though the property of being locally trivial
depends on the existence of cross local sections on the conjugation map \begin{align*}
\PUH &\to 
\homst{G}{\PUH}\\ 
F & \mapsto F\alpha F^{-1}. \end{align*}

Unfortunately such local cross sections fail to exist in general, as we shown in Theorem \ref{no sections homst}, and therefore the universal space that we can construct using families of subgroups
fails to be locally trivial. 
Nevertheless when we restrict ourselves to consider only maps $$B \to 
\homst{G}{\PUH}$$ with $B$ paracompact, we prove in Theorem \ref{theorem existence local section compact and connected for PUH} that these maps have indeed local
lifts to $\PUH$. The previous result would imply that the universal space constructed
using classifying spaces of families of subgroups done in \citet{LueckUribe} would
become a universal $G$ equivariant projective unitary stable bundle for paracompact spaces,
and hence, when restricted to paracompact spaces, 
the twisted equivariant K-theory would be a parametrized equivariant cohomology theory.
We have not proven this last statement, but we believe it is true.

Besides the application of our results to K-theory, we also show the following facts.
We study the space of stable unitary representations $\homst{G}{\UH}$ on
a Hilbert space $\HH$ and we generalize results of Dixmier-Douady 
 on the infinite grassmannian $\sg$ to the space of unitary representations.
 We show in Corollary \ref{corollary existence section compact and connected}
 that $\homst{G}{\UH}_{\mathcal C}$ is weakly homotopy equivalent to a point
 for any choice of irreducible representations $\CC \subset \Irrep(G)$, we show that
the space $\homst{G}{\PUH}$ has as many connected components as $S^1$-central
extensions of $G$ and that each connected component has $\hom(G,S^1)$ for fundamental group
and trivial higher homotopy groups. 

The article is organized as follows. In Section \ref{Chapter operators} we recall the properties
of the group of unitary operators endowed with the strong operator topology and we define the infinite
grassmannian. In Section \ref{Chapter continuous fields} we recall the definition of continuous field
of Hilbert spaces done by Dixmier-Douady in \citet{DixDou} and we show the properties
of the infinite grassmannian with respect to the existence of sections on the unitary group. In Section
\ref{Chapter unitary representations} we study the topological properties of the spaces
$$\homst{G}{\UH}$$ of stable continuous homomorphisms from a compact Lie group to the group of unitary 
operators. In Section \ref{Chapter projective unitary representations}
we study the topological properties of the space $$\homst{G}{\PUH}$$ of stable continuous homomorphisms from a compact
Lie group to the group of projective unitary operators. Finally, in Section \ref{Chapter applications}
we show some applications to twisted equivariant K-theory of the results of the previous sections
and we conclude with some ideas for further research.

{\bf Acknowledgements:} The first author acknowledges the support of a CONACyT postdoctoral fellowship and of the Centro de Ciencias Matem\'aticas of the UNAM. The second author acknowledges the financial support of the Max Planck Institute for Mathematics in Bonn and of COLCIENCIAS through contract number
FP44842-617-2014.

\footnotetext[1]{Corresponding author: Bernardo Uribe, bjongbloed@uninorte.edu.co. Received November ??, 2015. Accepted June 7, 2016}

\section{\small Spaces of operators and the infinite grassmannian} \label{Chapter operators}
Let $\HH$ be a separable and infinite dimensional complex Hilbert space and denote by $\BH$ the vector space of bounded linear operators. The inner product $\langle,\rangle$ on $\HH$ induces the norm $|x|:=\sqrt{\langle x,x \rangle}$ for  $ x \in \HH$, and we have the standard norm on the space of bounded linear operators $$|T| = \sup_{|x| \leq 1} |Tx|.$$
 The space $\BH$ can be endowed with several topologies, among them we have the \textit{strong operator topology} and the \textit{compact-open topology}. These topologies are of interest when studying
 principal bundles and  therefore are the ones of interest in this paper.

Recall that in the strong operator topology, a subbasic open set is given by $$V(T,\varepsilon;x) = \{ S\in \BH \mid |Sx-Tx|<\varepsilon\}$$
for any $T\in \BH$, $x\in \HH$ and $\varepsilon >0$. In this topology, a sequence of bounded operators $\{T_n\}$ converges to $T\in \BH$ if and only if $T_nx \to Tx$ for all $x\in \HH$. On the other hand, a subbase for the compact-open topology on $\BH$ is given by the family of sets
$$V(K,A) = \{S\in \BH \mid S(K) \subset A\},$$ 
where $K\subset \HH$ is a compact set and $A \subset \HH$ is an open set. Note that in the compact-open topology a sequence of bounded operators $\{T_n\}$ converges to $T\in \BH$ if and only if $T_n|_K \to T|_K$ uniformly for every compact set $K \subset \HH$.

Neither of the previous topologies on $\BH$ are equivalent.
 However, if we restrict to the group of unitary operators on $\HH$,
$$\UH=\{ U \in \BH \mid U U^*=U^*U= \rm{Id}_\HH \},$$
then we know that the strong operator topology and the compact-open topology agree on the group $\UH$.
 Moreover, the group $\UH$ endowed with any of these topologies is a Polish group, i.e. a completely
  metrizable topological group (see \cite{Espinoza-Uribe:14}), and furthermore contractible \cite[\S11, Lem. 3]{DixDou}.
Let us recall the metric on $\UH$ which recovers the strong operator topology since 
it will be needed in what follows.

Let $\{e_j\}_{j \in \mathbb N}$ be an orthonormal base of $\HH$ and consider the embedding $\Psi: \UH \to \HH^{\mathbb N}$ with $\Psi(T)= ( Te_j)_{j \in \mathbb N}$. Any metric
on $\HH^{\mathbb N}$ which induces the product topology induces also a metric on $\UH$
compatible with the strong operator topology.
Therefore for any pairs of operators $T, U \in \UH$ we may define their distance by the equation
\begin{eqnarray} \label{metric UH}
\langle T, U \rangle := \sup_{n \in \mathbb N} \left\{ \frac{\min\{|Te_n - U e_n|, 2\}}{n} \right\}.
  \end{eqnarray}  
Note that with this metric $S^1 $ acts by isometries on $\UH$, i.e.
\begin{eqnarray*}
\langle e^{it}T, e^{it}U \rangle= \langle T, U \rangle
\end{eqnarray*}
and moreover we have that
\begin{eqnarray*}
\langle e^{it}T, e^{is}T \rangle= |e^{it}-e^{is}|.
\end{eqnarray*}

From now and on we will assume that $\UH$ is endowed with the strong operator topology
and with the metric defined above. The first 
consequence of this choice is the following lemma.

\begin{lemma} \label{evaluation-adjuncion}
Let $X$ be a topological space and let us consider $\UH$ with the strong operator topology. Then a map $\Psi : X \to \UH$ is continuous if and only if the map
\begin{eqnarray*}
\psi : X \times \HH & \too & \HH \\
(x,h) & \mapstoo & \Psi(x)h
\end{eqnarray*}
is continuous.
\end{lemma}

\begin{proof}
Consider $(x_0, h_0) \in X\times \HH$ and let $$B(\Psi(x_0)h_0,\varepsilon)$$ be an open ball in 
$\HH$ with center at $\psi(x_0,h_0)=\Psi(x_0)h_0$ and radius $\varepsilon>0$. We will show
that there exists an open set $A\subset X \times \HH$ such that $\psi(A)\subset B(\Psi(x_0)h_0,\varepsilon)$.

Define $ V(\Psi(x_0),\varepsilon/2; h_0) = \{T\in \UH \mid Th_0\in B(\Psi(x_0)h_0,\varepsilon/2) \}$ which is an open set in the strong operator topology on $\UH$. Then $U = \Psi^{-1}(V(\Psi(x_0),\varepsilon/2; h_0))$ is an open set in $X$.

Let us see that the open set $$A = U \times B(h_0;\varepsilon/2) \subset X \times \HH$$ has the desired properties. Indeed, for $(x,h)\in A$ we have
\begin{align*}
|\psi  (x, & h) - \psi(x_0,h_0)| \\
 = & |\Psi(x)h - \Psi(x_0)h_0| \\
 \leq & |\Psi(x)h - \Psi(x)h_0| + |\Psi(x)h_0 - \Psi(x_0)h_0| \\
 = & |h-h_0| + |\Psi(x)h_0 - \Psi(x_0)h_0| <  \frac{\varepsilon}{2} + \frac{\varepsilon}{2},
\end{align*}
and therefore $A\subset \psi^{-1} \left( B(\Psi(x_0)h_0,\varepsilon) \right)$.

By \cite[Thm. 46.11]{Munkres:Top} the reciprocal statement is true whenever $\UH$ is endowed with the compact-open topology. Since the strong operator topology and the compact-open
 topology agree on $\UH$ the lemma follows.

\end{proof}

Recall that an orthogonal projector $P$ on the Hilbert space consist of an operator $P: \HH \to \HH$ such that $P^2=P$ with the property that $\ker(P)$ and $\im(P)$ are orthogonal. Define the \textit{infinite grassmannian} as follows
\begin{align*}
\sg  :=   \{P \in  \mathcal{B}(\HH) & \mid P^2=P,  \ker(P)= \im(P)^\perp, \\   & \dim(\ker(P)) = \dim(  \im(P))= \infty \}\end{align*}
and endow it with the strong operator topology.

Since the map
$$ \sg \to \UH, \ \ P \mapsto 2P-1$$
is an embedding, and $\UH$ is a Polish group, then $\sg$ is metrizable. Moreover, since $\sg$ is closed in $\UH$
we have that the infinite grassmannian $\sg$ is a completely metrizable space.

\section{Continuous fields of Hilbert spaces} \label{Chapter continuous fields}

The following is the definition of Dixmier and Douady \cite{DixDou} applied to the case of Hilbert spaces.

Consider $B$ a topological space  and denote by $\mathcal{O}(B)$ the algebra of continuous complex valued functions
on $B$. Let $(E(z))_{z \in B}$ be a family of Hilbert spaces. For $Y \subset B$, an element in $\prod_{z \in Y}E(z)$, i.e. 
an assignment  $s$ defined on $Y$ such that $s(z) \in E(z)$ for all $z \in Y$, will be called a {\it vector field} over $Y$.  For $s$ a vector field
over $Y$, we will denote by $\|s\|$ the map $z \mapsto \|s(z)\|$ taking values in $\mathbb{R}$.

\begin{definition} \cite[Def. 1, pp 228]{DixDou}
A {\it continuous field of Hilbert spaces} $\mathscr{E}$ over the topological space $B$ is a family $(E(z))_{z \in B}$ 
of Hilbert spaces, endowed with a set $\Gamma \subset \prod_{z \in B}E(z)$ of vector fields, such that:

\begin{itemize}
\item $\Gamma$ is a sub-$\mathcal{O}(B)$-module of $ \prod_{z \in B}E(z)$.
\item For all $z \in B$ and all $\xi \in E(z)$, there exists $s \in \Gamma$ such that $s(z)= \xi$.
\item For all $s \in \Gamma$, the map $\|s\|$ is continuous.
\item If $s \in  \prod_{z \in B}E(z)$ is a vector field such that for all $z \in B$ and every $\varepsilon >0$ there exists $s' \in \Gamma$
 satisfying $\|s - s'\| \leq \varepsilon$ on a  neighborhood of $z$, then $s \in \Gamma$.
\end{itemize}
\end{definition}

The elements of $\Gamma$ will be called {\it continuous vector fields} of $\mathscr{E}$.

Let $\HH$ be a Hilbert space and $\Gamma$ the space of continuous maps from $B$ to $\HH$. For every
$z \in B$ define $E(z):= \HH$. Then $\mathscr{E}= ((E(z))_{z \in B}, \Gamma)$ is a continuous 
field of Hilbert spaces and will be called the {\it constant field} over $B$ defined by $\HH$.

Following \cite[\S12]{DixDou}, denote by $ \mathscr{D}_0$ the constant field over $\sg$ defined by $\HH$ and denote by $\mathscr{D}$
the {\it canonical field} over $\sg$ defined by the family of vector spaces $(P(\HH))_{P \in \sg}$, where $\Gamma$ is the
set of vector fields parametrized by the elements in $\HH$
$$\Gamma = \{ \bar{\xi} \in \prod_{P \in \sg} P(\HH) \mid \xi \in \HH \}$$
with $ \bar{\xi}(P):=P(\xi)$. 
Denote by $\mathscr{D}^\perp$ the family of vector spaces $(P(\HH)^\perp)_{P \in \sg}$ and note that
both $\mathscr{D}$ and $\mathscr{D}^\perp$ are sub-fields of $\mathscr{D}_0$ and moreover $\mathscr{D} \oplus \mathscr{D}^\perp \cong \mathscr{D}_0$.

From Theorem 2 in \cite[\S12]{DixDou}, it follows that the canonical field over the infinite grassmannian is trivial if $\sg$ has the norm topology.
Nevertheless whenever $\sg$ is endowed with the strong operator topology, then the canonical field is not locally trivial \cite[\S16, Cor. 2]{DixDou}.

The continuous field $\mathscr{D}$ over $\sg$ is a universal continuous field for 
continuous fields of infinite dimensional Hilbert spaces over paracompact spaces. This fact follows from \cite[\S14, Cor. 1]{DixDou}
which we quote here: Let $\mathscr{E}$ be a continuous field of infinite dimensional and separable Hilbert spaces
over the paracompact space $B$. Then there exist a continuous map $\phi : B \to \sg$ such that $\mathscr{E}\cong \phi^*\mathscr{D}$.

Dixmier and Douady show this fact in \cite[\S13, Thm. 3]{DixDou} by constructing vector fields $\{ \bar{s}_n \}_{n \in \mathbb{N}} \subset \Gamma$
which are orthogonal over $\phi(B)$ and such that for all $P \in \phi(B)$ the set
$\{ \bar{s}_n(P) \}_{n \in \mathbb{N}}$ is an orthonormal base for $P(\HH)$. With these sections
at hand Dixmier and Douady furthermore show in \cite[\S15, Thm. 5]{DixDou}  that $\mathscr{E}$
 is trivializable.
 
 \begin{lemma} \label{lemma:sections of UH on paracompact B}
Let $B$ be a paracompact space of finite covering dimension and $\phi : B \to \sg$ a continuous map. Take $b_0 \in B$ and define the
conjugation map 
\begin{align*}
\pi_{\phi(b_0)} :\UH &  \too \sg \\
F & \mapstoo  F\phi(b_0)F^{-1}.
\end{align*}
Then there exist a continuous map $\sigma : B\to \UH$ such that the following diagram is commutative
$$\xymatrix{
&& \UH \ar[d]^-{\pi_{\phi(b_0)}} \\
B \ar[rru]^-{\sigma} \ar[rr]^-{\phi} && \sg.
}$$
\end{lemma}
\begin{proof}
Since $\phi^*\D \oplus \phi^*\D^\perp \cong B \times \HH$ and  $\phi^*\D$ is trivializable, in \cite[\S14, Thm. 4]{DixDou} it is shown that there exist sections $\{ \bar{s}_n \}_{n \in \mathbb{N}}$,
$\{ \bar{t}_n \}_{n \in \mathbb{N}}$ for $\phi^*\D$ and $\phi^*\D^\perp$ respectively
which are pointwise orthonormal, i.e. $\{ \bar{s}_n(b) \}_{n \in \mathbb{N}}$ is an orthonormal base for
$\phi(b)(\HH)$ and $\{ \bar{t}_n(b) \}_{n \in \mathbb{N}}$ is an orthonormal base for
$({\rm{Id}}-\phi(b))(\HH)$. 

Define the map $\psi : B \times \HH \to \HH$ by the assignment 
$$\psi(b, \bar{s}_n(b_0)) = \bar{s}_n(b) \ \ \mbox{and} \ \ 
\psi(b, \bar{t}_n(b_0)) = \bar{t}_n(b)$$
and note that $\psi$ is continuous. By Lemma \ref{evaluation-adjuncion} the map $\sigma: B \to \UH$
defined by the equation $\psi(b,h)=\sigma(b)h$ is continuous  and we have that
$$\sigma(b) \bar{s}_n(b_0)=\bar{s}_n(b) \ \ \mbox{and} \ \
\sigma(b) \bar{t}_n(b_0)=\bar{t}_n(b).$$
Therefore \begin{align*}
\sigma(b) \phi(b_0) \left( \bar{s}_n(b_0)\right) = & \sigma(b)  \left( \bar{s}_n(b_0)\right) \\
=&    \bar{s}_n(b) \\
=& \phi(b)( \bar{s}_n(b)) \\
=&  \phi(b) \sigma(b) (\bar{s}_n(b_0))
\end{align*}
and $\sigma(b) \phi(b_0) \left( \bar{t}_n(b_0)\right) =0 = \phi(b) \sigma(b) \left( \bar{t}_n(b_0)\right)$. Therefore $\sigma(b) \phi(b_0)= \phi(b) \sigma(b)$ and
we have the desired equation $\sigma(b) \phi(b_0) \sigma(b)^{-1}= \phi(b)$.
\end{proof}

We have just shown that over paracompact spaces we may construct sections of the conjugation
map. But these sections fail to exist whenever the base is the infinite grassmannian.

\begin{lemma}\label{no global sections for grassmannian}
Let $Q\in \sg$ be a projector. Then, the map 
\begin{eqnarray*}
\pi_Q:\UH & \too & \sg \\
F & \mapstoo & FQF^{-1}
\end{eqnarray*}
has no global sections.
\end{lemma}

\begin{proof}
Let us proceed by contradiction.
Suppose that there exists a continuous map $\sigma : \sg \to \UH$ such that $\sigma(P)Q\sigma(P)^{-1} = P$. By Lemma \ref{evaluation-adjuncion} we know that the evaluation map $(P, h) \mapsto \sigma(P)h $ is continuous. Hence the map $(P, h) \mapsto (P, \sigma(P)h)$ is also continuous, and its restriction
\begin{eqnarray*}
\sg \times \im(Q) &\too& \D  \\
(P, h) & \mapstoo & (P, \sigma(P)h)
\end{eqnarray*}
is continuous. This map trivializes the canonical field $\D$ over $\sg$ but this contradicts \cite[\S16, Cor. 2]{DixDou} where it is shown that $\D$ is nowhere locally trivial. 
\end{proof}

The map $\pi_Q$ is surjective in $\sg$ since any two orthogonal projectors in $\sg$ are conjugate. Therefore
the map $\pi_Q$ induces a continuous map
$$\UH/ \UH_Q \too \sg, \ \ \ [F] \mapsto FQF^{-1}$$
where $\UH_Q=\{T \in \UH \mid TQT^{-1}=Q \}$ is the isotropy group of $Q$, which is moreover bijective 
but which is not a homeomorphism. This last statement follows from \cite[\S16, Cor. 2]{DixDou}
where it is shown that $\D$ is nowhere locally trivial thus implying that the map $\pi_Q$ is not
a $\UH_Q$-principal bundle over $\sg$.

Nevertheless, the existence of extensions that were shown in Lemma \ref{lemma:sections of UH on paracompact B}
implies that the pullback $\phi^*\UH$ over $B$ of a fixed  map $\phi: B \to \sg$ is indeed a $\UH_{\phi(b_0)}$-principal bundle. If $\phi^*\UH= \{(b, F) \in B \times \UH \mid \phi(b)=\pi_{\phi(b_0)}(F) \}$
then the map
$$B \times \UH_{\phi(b_0)} \too \phi^*\UH, \ \ \ (b,T) \mapsto (b,\sigma(b)T)$$
is a $\UH_{\phi(b_0)}$-bundle isomorphism.

We conclude this section by stating that $\sg$ is a classifying space for Hilbert space bundles
over paracompact spaces. This follows from the following three facts: First, any continuous field of infinite dimensional Hilbert
spaces over a paracompact space $B$ is isomorphic to the pullback over some map
 of the canonical field $\D$ over $\sg$. Second, any continuous field of infinite dimensional
 Hilbert spaces over a paracompact space is trivial. Third, the infinite grassmannian $\sg$ is contractible.
 Hence any two maps from $B$ to $\sg$ are homotopic, and any two continuous fields of infinite
 dimensional spaces over $B$ are isomorphic.

\section{\small Spaces of unitary representations} \label{Chapter unitary representations}

Let $G$ be a compact Lie group and consider a continuous homomorphism $\alpha: G \to \UH$. 
The homomorphism $\alpha$ induces the structure of a $G$ representation to $\HH$ denoted by $\HH_\alpha$ and
we have a canonical decomposition of $\HH_\alpha$ in isotypical components $$\HH \cong \bigoplus_{V \in \Irrep(G)}\HH_\alpha^V$$
where $\Irrep(G)$ denotes the isomorphism classes of irreducible representations of $G$, $V$
is a representative of its isomorphism class of irreducible representation and 
$\HH_\alpha^V$
is the isotypical subspace associated to $V$.  By Schur's Lemma, the isotypical part associated to $V$
may be defined as the image of the evaluation map, i.e.
\begin{align}\label{definition of evaluation map}
\HH_\alpha^V= \im\left(ev: V \otimes \hom_G(V, \HH_\alpha) \to \HH_\alpha \right).\end{align}
with $ev(v \otimes f)=f(v)$.
\begin{definition}
Let $\alpha: G \to \UH$ be a continuous homomorphism from a compact Lie group $G$ to $\UH$. We say that
the homomorphism is {\it{stable}} if all the isotypical components of $\HH_\alpha$ are either
infinite dimensional or zero dimensional. We will denote the set of stable homomorphisms as $\homst{G}{ \UH}$.
\end{definition}

The set of all homomorphisms $\hom(G,\UH)$ can be endowed with the subspace topology
of the compact-open topology of the space $$\map(G, \UH)$$ of continuous maps from $G$ to $\UH$, and a subbase for this topology is given by the family of sets
\begin{align*}
&\mathcal V(a, \varepsilon; x) := \\ &\{ \gamma \in \hom (G, \UH) \mid \|\gamma(g) - a(g)x \| <\varepsilon, \ g\in G\}\end{align*}
for $a\in \hom(G,\UH)$, $\varepsilon>0$ and $x \in \HH$.

On the other hand, since $G$ is compact and $\UH$ is metrizable, we may also endow  $$\map(G, \UH)$$ with the supremum metric, i.e. 
for $$\alpha, \beta \in \map(G, \UH)$$ we define $$d_{sup}(\alpha,\beta)= \sup\{d(\alpha(g),\beta(g)) \mid g \in G \}.$$
By \cite[Thm. 46.8]{Munkres:Top} we know that these two topologies agree. Moreover,  
since $\UH$ is complete  we know that
$\map(G, \UH)$ is also complete \cite[Thm. 43.5]{Munkres:Top}.

\begin{lemma}
The space $\hom(G,\UH)$ is a complete metric space.
\end{lemma}
\begin{proof}
Since $\map(G,\UH)$ is complete we may take a convergent sequence $\{\alpha_n\}_{n \in \mathbb N}$
of homomorphisms which converge to $\alpha \in \map(G, \UH)$. We know that
$$\alpha_n(g)\alpha_n(h) \to \alpha(g)\alpha(h)$$ and $\alpha_n(gh) \to \alpha(gh)$, and since
 $\alpha_n(g)\alpha_n(h)=\alpha_n(gh)$ we conclude that $$\alpha(g)\alpha(h)=\alpha(gh).$$
 Therefore $\alpha$ is also a homomorphism and hence $\hom(G,\UH)$ is complete.
\end{proof}

\begin{lemma} The space of stable homomorphisms $\homst{G}{\UH}$ is not closed in $\hom(G,\UH)$.
\end{lemma}
\begin{proof}
We will show a convergent sequence in $\homst{S^1}{\UH}$ whose limit is not stable. The argument
for any compact Lie group is similar.

Let $\{e_j\}_{j \in \mathbb N}$ be an orthonormal base of $\HH$. Define the homomorphisms
$\alpha_k : S^1 \to \UH$ by the assignment $$\alpha_k(e^{i\theta})(e_{2j})=e_{2j}$$
and
$$  \alpha_k(e^{i\theta})(e_{2j+1})=  \left\{
 \begin{array}{ccl}
 e_{2j+1} & \mbox{if} & 0< j <k\\
 e^{i \theta} e_{2j+1} & \mbox{if} & j \geq k \ \mbox{or} \ j=0
 \end{array} \right.$$
and note that the $\alpha_k$'s are stable. Since $$\lim_{k \to \infty} \alpha_k(e^{i \theta})e_j = e_j $$ for
$j>1$ and $$\lim_{k \to \infty} \alpha_k(e^{i \theta})e_1 = e^{i \theta}e_1$$
we know that $\lim_{k \to \infty} \alpha_k $ does not belong to the space of stable homomorphisms.
\end{proof}

\begin{lemma} The space $\homst{G}{\UH}$ of stable homomorphisms is not open in $\hom(G,\UH)$.
\end{lemma}
\begin{proof}
We will prove that for any basic open set $\mathcal V \subset \hom (G, \UH)$ and $a\in \mathcal{V}$ a stable homomorphism, there exist a non stable homomorphism $b \in \mathcal V$.

Fix $\varepsilon>0$ and $x_1,\ldots,x_n \in \HH$. Define
\begin{align*}
\mathcal V   = & \mathcal V(a, \varepsilon; x_1,\ldots,x_n) \\
  =& \{ \gamma \in \hom (G, \UH) \mid & \Vert \gamma(g)x_k - a(g)x_k \Vert <\varepsilon, \\
   & & \forall g\in G, \forall k=1,\ldots, n \} 
\end{align*}
and let $ H= V_1 \oplus \cdots \oplus V_n$ be the direct sum of the irreducible representations $V_k$ of $G$, such that $x_k \in V_k$ for each $k=1, \ldots, n$. It follows that $H$ is $a$-invariant and finite dimensional.

Let $b : G \to \UH$ be given by
$$b(g)x = 
\begin{cases}
a(g)x & x\in H, \\
x & x\in \HH \ominus H.
\end{cases}
$$
Then $b \in \mathcal V$ by construction, but $a$ and $b$ are not unitary equivalent, i.e. $b$ is not stable.
\end{proof}

\begin{lemma}
The space $\homst{G}{\UH}$ of stable homomorphisms is dense in $\hom(G,\UH)$.
\end{lemma}
\begin{proof}

Let $\alpha \in \hom (G, \UH)$ be a homomorphism and let $\mathcal V(\alpha, \varepsilon; x_1,\ldots,x_n)$ be a basic open set based at $\alpha \in \mathcal{V}$. Consider the finite dimensional and $\alpha$-invariant space $ H = V_1 \oplus \cdots \oplus V_n$ given by the direct sum of the irreducible representations $V_k$ of $G$, such that $x_k \in V_k$ for each $k = 1, \ldots, n$. If $\psi$ is an isometric isomorphism $\psi: L^2(G)\otimes L^2([0,1])\to \HH \ominus H$, then 
$$b(g)x = 
\begin{cases}
\alpha (g)x & x\in H\\
\psi \alpha(g) \psi^{-1}x & x\in \HH \ominus H
\end{cases}
$$
is a stable homomorphism and $b \in \mathcal{V}$.
\end{proof}

\begin{definition} 
Let $\CC \subset \Irrep(G)$ be a choice of irreducible representations of the group $G$. Define
\begin{align*}
\hom(G,\UH)_\CC:= & \\
 \{ \alpha \in \hom(G,\UH) & \mid\dim(\HH_\alpha^V)=0 \ \mbox{for} \ V \notin \CC \}\end{align*} to be the space of homomorphisms which induce representations where only the irreducible 
representations in $\CC$ appear. Define
\begin{align*}
\homst{G}{\UH}_\CC := & \\ 
\hom(G, &\UH)_\CC \bigcap \homst{G}{\UH}.\end{align*}
\end{definition}
We have therefore
\begin{align*}
&\homst{G}{\UH}_\CC:= \\
&\{ \alpha \in  \hom(G,\UH) \mid\dim(\HH_\alpha^V)=0 \ \mbox{for} \ V \notin \CC \\
 & \ \mbox{and} \ \dim(\HH_\alpha^V)=\infty \ \mbox{for} \ V \in \CC  \}.\end{align*}

The spaces $\homst{G}{\UH}_\CC$ are neither closed nor open, nevertheless the action by
 conjugation of $\UH$ on $\homst{G}{\UH}_\CC$ is transitive and we are interested in studying the 
 properties of this action.
 
 \begin{definition}
 Take a stable homomorphism $\alpha \in \homst{G}{\UH}_\CC$. Define the conjugation map
 $$\pi_\alpha: \UH \to \homst{G}{\UH}_\CC, \ \ F \mapsto F\alpha F^{-1}.$$
\end{definition}

The map $\pi_\alpha$ is continuous and is surjective. Any other stable 
homomorphism $$\alpha' \in \homst{G}{\UH}_\CC$$ induces an isomorphism $\HH \cong \bigoplus_{V \in \CC}\HH_{\alpha'}^V$. For each $V \in \CC$ choose a $G$-equivariant unitary isomorphism $F^V: \HH_\alpha^V \stackrel{\cong}{\to}
\HH_{\alpha'}^V$ and assemble them into a $G$-equivariant unitary automorphism 
$$F:= \bigoplus_{V \in \CC}F^V : \HH \stackrel{\cong}{\to} \HH.$$  By construction the unitary
automorphism satisfies $F  \alpha(g) = \alpha'(g) F$ for all $g \in G$, and therefore $F\alpha F^{-1}
=\alpha'$.

The surjectivity of the map $\pi_\alpha$ implies that the map
\begin{align*}
\UH/\UH_\alpha  & \too \homst{G}{\UH}_\CC \\ [F] & \mapsto F\alpha F^{-1}\end{align*}
is a continuous map which is moreover bijective, where $\UH_\alpha:= \{T\in \UH \mid T \alpha T^{-1}=\alpha\}$ is the isotropy group of $\alpha$. We will show that this map is not a homeomorphism,
thus implying that the $\pi_\alpha$ is not a $\UH_\alpha$-principal bundle. Nevertheless, the 
pullback of $\pi_\alpha$ for maps defined on paracompact spaces is indeed a $\UH_\alpha$-principal
bundle.

\begin{theorem} \label{no sections homst}
Suppose that $\CC$ contains the trivial representation, then
the conjugation map $\pi_\alpha: \UH \to \homst{G}{\UH}_\CC$ has no sections.
\end{theorem}
\begin{proof}
Note that in the case that $G=\mathbb Z/2\mathbb Z$  we have a homeomorphism
$$\homst{\mathbb Z/2\mathbb Z}{\UH}\stackrel{\cong}{\to} \sg $$
where the homomorphism $$\beta  \in \homst{\mathbb Z/2\mathbb Z}{\UH}$$ is mapped
to the orthogonal projector $$\frac{1}{2}\left(\beta(1) + \beta(-1)\right).$$
 The same proof of Lemma \ref{no global sections for grassmannian} shows that 
 the map $\pi_\alpha: \UH \to \homst{\mathbb Z/2\mathbb Z}{\UH}$ has no sections.
 
 The proof of the general case is based on the non existence of sections for the canonical field over the infinite grassmannian. 
 We just need to find an appropriate injective map from $\sg$ to  $\homst{G}{\UH}_\CC$.
 
Choose a representation $V \in \CC$ different from the trivial representation.
 Take the isotypical decomposition of $\HH \cong \bigoplus_{W \in \CC} \HH_\alpha^W$ defined by $\alpha$. Consider the infinite grassmannian
 $$\Gr(\hom_G(V,\HH_\alpha^V))$$
 of the Hilbert space $\hom_G(V,\HH_\alpha^V)$ and for each $Z \in \Gr(\hom_G(V,\HH_\alpha^V))$
 denote by $ev(V \otimes Z)$ the subspace of $\HH_\alpha^V$ defined by the image of $V \otimes Z$
 under the evaluation map  $$ev: V \otimes \hom_G(V,\HH_\alpha^V) \to \HH_\alpha^V$$ with $ev(v \otimes f)=f(v)$ that was previously defined in \eqref{definition of evaluation map}.
 
 Denote by $ev(V \otimes Z)^{\perp_{\HH_\alpha^V}}$ the orthogonal complement of $ev(V \otimes Z)$ in $\HH_\alpha^V$, i.e  $$ev(V \otimes Z) \oplus  ev(V \otimes Z)^{\perp_{\HH_\alpha^V}} \cong \HH_\alpha^V$$
 and define the map
 \begin{align*}
 \Phi:\Gr(\hom_G(V,\HH_\alpha^V)) & \to \homst{G}{\UH}_\CC \\ Z & \mapsto
 \Phi(Z)
 \end{align*} with
 $$\Phi(Z)(g)(v\oplus v') :=  \alpha(g)v \oplus v'$$
 where $$v \in \left( ev(V \otimes Z) \oplus \bigoplus_{W \in \CC, W \neq V} \HH_\alpha^V \right),$$  $v' \in ev(V \otimes Z)^{\perp_{\HH_\alpha^V}}$ and $g \in G$.

Note that the homomorphism $\Phi(Z)$ only disagrees with $\alpha$ on the subspace $ev(V \otimes Z)^{\perp_{\HH_\alpha^V}}$,
that the isotypical subspace of the homomorphism $\Phi(Z)$ associated to $V$ is precisely $ev(V \otimes Z)$, i.e.
$$\HH_{\Phi(Z)}^V = ev(V \otimes Z), $$
and that
\begin{align*}
\hom_G(V, \HH_{\Phi(Z)}^V) =Z.
\end{align*}

 The map $\Phi$ is continuous since it can be defined as the composition of projections, and it is moreover
 injective. 
 
 Choose a base point $Z_0 \in \Gr(\hom_G(V,\HH_\alpha^V))$ and let us suppose that the map 
 $\pi_{\Phi(Z_0)} : \UH \to  \homst{G}{\UH}_\CC$, $ F \mapsto F \Phi(Z_0) F^{-1}$ has a section 
 $\sigma: \homst{G}{\UH}_\CC \to \UH$. Hence we would have that for all $g \in G$ the following equality
 holds
 $$\sigma(\Phi(Z)) \left( \Phi(Z_0)(g) \right) = \left(\Phi(Z)(g)\right) \sigma(\Phi(Z)),$$
which implies that $\sigma(\Phi(Z))$ induces a $G$-equivariant unitary isomorphism 
 between the isotypical components
 $$\sigma(\Phi(Z))|_{\HH_{\Phi(Z_0)}^W} : {\HH_{\Phi(Z_0)}^W} \stackrel{\cong}{\to} {\HH_{\Phi(Z)}^W}$$
 and in particular it induces a $G$-equivariant isomorphism
 $$\sigma(\Phi(Z))|_{ev(V \otimes Z_0)} : ev(V \otimes Z_0) \stackrel{\cong}{\to} ev(V \otimes Z).$$
 
 Denote by $\D$ the canonical continuous field of Hilbert spaces over $\Gr(\hom_G(V,\HH_\alpha^V))$
 defined as follows:
  \begin{align*}
  \D & := \{(Z ,f) \in  \\
  & \Gr(\hom_G(V,\HH_\alpha^V)) \times \hom_G(V,\HH_\alpha^V) \mid f \in Z \},\end{align*}
 and construct the following map:
 \begin{align*}
 \theta: \Gr(\hom_G(V,\HH_\alpha^V)) \times Z_0 & \to \D \\ (Z,f_0) & \mapsto (Z, \sigma(\Phi(Z)) \circ f_0).\end{align*}
 
 Note that the  homomorphism $\sigma(\Phi(Z)) \circ f_0$ belongs to $Z$, since 
 the image of $f$ lies in ${\HH_{\Phi(Z_0)}^V}$ and the unitary isomorphism
 $$\sigma(\Phi(Z))|_{\HH_{\Phi(Z_0)}^V} : {\HH_{\Phi(Z_0)}^V} \stackrel{\cong}{\to} {\HH_{\Phi(Z)}^V}$$ is $G$-equivariant; hence $$\sigma(\Phi(Z)) \circ f_0 \in  \hom_G(V,\HH_{\Phi(Z)}^V) =Z.$$ Moreover the map $\theta$ is continuous since $\sigma$ is continuous, and its inverse map
 is simply $\theta^{-1}(Z,f) = (Z, \sigma(\Phi(Z))^{-1}f)$.
 Therefore $\theta$ would be a trivialization of the canonical field 
 $\D$ which contradicts \cite[\S16, Cor. 2]{DixDou} where it is shown that $\D$ is nowhere locally trivial. 
 
 Then the section $\sigma$ cannot exist and the theorem follows.
\end{proof}

\begin{theorem} \label{theorem existence of sections abelian compact}
 Let $K$ be an abelian  compact Lie group, $B$ a paracompact
space of finite covering dimension, with base point $b_0 \in B$ and $f: B \to \homst{K}{\UH}_\CC$ a continuous map
with $\CC \subset \Irrep(K)$. Then there exist an extension
$\sigma: B \to \UH$ that makes the following diagram commutative
$$\xymatrix{
&& \UH \ar[d]^{\pi_{f(b_0)}} \\
B \ar[rr]^-{f} \ar[rru]^-{\sigma} & & \homst{K}{\UH}_\CC
}$$
\end{theorem}

\begin{proof}
Since all irreducible representations of $K$ are 1-dimensional, let us encode the information
of each irreducible $V \in \Irrep(K)$ by a homomorphism $\chi_V: K \to S^1 \subset \mathbb C^*$.
Choose a normalized left and right invariant measure on $K$, and for any $\alpha \in 
\homst{K}{\UH}_\CC$  define the operator
$$\psi_V(\alpha) : \HH \to \HH, \ \ h \mapsto \int_K \alpha(k) \chi_V(k^{-1}) h\, \mathrm{d}k. $$ 

Whenever $h \in \HH_\alpha^V$ we have that $\psi_V(\alpha)h =h$, and whenever
$h \in  \HH_\alpha^W$ for $W \neq V$ then $\psi_V(\alpha)h=0$. Therefore the operator
$\psi_V(\alpha)$ is equivalent to the orthogonal projector $P_{\HH_\alpha^V}$ that
projects $\HH$ to $\HH_\alpha^V$. For $V \in \CC$  the assignment $\psi_V$ defines a map
\begin{align*}
\psi_V: \homst{K}{\UH}_\CC & \to \sg \\
 \alpha & \mapsto \psi_V(\alpha)=P_{\HH_\alpha^V}\end{align*}
which is continuous since the integration is over a compact Lie group.

Consider the composition $\psi_V \circ f : B \to \sg$ and note that $(\psi_V \circ f)^* \D$
is trivializable. By  \cite[\S14, Thm. 4]{DixDou} we know that there exist orthogonal sections $\{ {s}^V_n \}_{n \in \mathbb N}$
of $(\psi_V \circ f)^* \D$ which satisfy $\bar{s}^V_n(b) \in \HH_{\psi_V(f(b))}^V$
and that moreover the vectors $\{{s}^V_n(b) \}_{n \in \mathbb N}$ are an orthogonal base of $\HH_{\psi_V(f(b))}^V$.

Define the map $\phi : B \times \HH \to \HH$ by the assignment 
$$\phi(b, {s}^V_n(b_0)) = {s}^V_n(b)$$
where $V$ runs over the irreducible representations in $\CC$ and $\HH$ is viewed
as $\HH \cong \bigoplus_{V \in \CC} \HH_{f(b_0)}^V$ in the source and as
$\HH \cong \bigoplus_{V \in \CC} \HH_{f(b)}^V$ in the target. Since the sections $s^V_n$ are continuous, the map $\phi$ is continuous.

By Lemma \ref{evaluation-adjuncion} the map $\sigma: B \to \UH$
defined by the equation $\sigma(b)h :=\phi(b,h)$ is continuous  and we have that
$$\sigma(b) {s}^V_n(b_0)={s}^V_n(b)$$
for all $V \in \CC$. In particular we have that $$\sigma(b) P_{\HH_{f(b_0)}^V} \sigma(b)^{-1}=
P_{\HH_{f(b)}^V}$$ and therefore $$\sigma(b) \psi_V(f(b_0)) \sigma(b)^{-1}= \psi_V(f(b))$$
for all $V \in \CC$.

Define the map
\begin{align*}
\Psi :  \homst{K}{\UH}_\CC & \to \prod_{V \in \CC}\sg \\ 
\alpha & \mapsto \prod_{V \in \CC} \psi_V(\alpha)\end{align*}
and note that 
the argument above implies that $\sigma$ makes the following diagram commutative
$$\xymatrix{ && \UH \ar[d]^{\prod_{V \in \CC} \pi_{\psi_V(f(b_0))}} \\
B \ar[rr]^{\Psi \circ f} \ar[urr]^{\sigma} && \prod_{V \in \CC} \sg }$$
where $\pi_{\psi_V(f(b_0))} : \UH \to \sg$ is the conjugation map $\pi_{\psi_V(f(b_0))}(F)=F 
\psi_V(f(b_0)) F^{-1}$.

Finally note that the map $$\Psi : \homst{K}{\UH}_\CC \to \prod_{V \in \CC}\sg$$ is injective
since for abelian groups the isotypical spaces determine the homomorphism. Therefore the following
diagram is commutative 
$$\xymatrix{ &\UH \ar[r]^= \ar[d]_{\pi_{f(b_0)}}& \UH \ar[d]_{\prod_{V \in \CC} \pi_{\psi_V(f(b_0))}} \\
B \ar[r]_-{f} \ar[ur]^{\sigma} & \homst{K}{\UH}_\CC \ar@{^(->}[r]_{\Psi}  &\prod_{V \in \CC} \sg }$$
and the theorem follows.
\end{proof}

\begin{corollary} \label{corollary existence section compact and connected}
 Let $G$ be a compact and connected Lie group, $B$ a paracompact
space of finite covering dimension, with base point $b_0 \in B$ and $f: B \to \homst{G}{\UH}_\CC$ a continuous map
with $\CC \subset \Irrep(G)$. Then there exist an extension
$\sigma: B \to \UH$ that makes the following diagram commutative
$$\xymatrix{
&& \UH \ar[d]^{\pi_{f(b_0)}} \\
B \ar[rr]^-{f} \ar[rru]^\sigma& & \homst{G}{\UH}_\CC
}$$
\end{corollary}
\begin{proof}
Let $K$ be a maximal abelian subgroup of $G$  and denote by $\iota: K \to G$ the inclusion.
Recall that any representation of the group $G$ is uniquely determined by its restriction
to a maximal abelian subgroup, and therefore the restriction map
$$ \iota^*:\homst{G}{\UH}_\CC \hookrightarrow  \homst{K}{\UH}_{\iota^*\CC}$$
is injective. Here we have denoted by $\iota^*\CC \subset \Irrep(K)$ 
the set of irreducible representations in $K$ which appear as restrictions of representations $V$ of $\CC$.

By Theorem \ref{theorem existence of sections abelian compact} we know that there 
exists $\sigma$ for the composition map$$\iota^* \circ f : B \to \homst{K}{\UH}_{\iota^*\CC},$$
and since $\iota^*$ is injective, the following diagram commutes
$$\xymatrix{& \UH \ar[d]^{\pi_{f(b_0)}} \ar[r]^= & \UH \ar[d]^{\pi_{\iota^*(f(b_0))}}\\
B \ar[r]^-f \ar[ru]^\sigma & \homst{G}{\UH}_\CC  \ar@{^(->}[r]^{\iota^*}&  \homst{K}{\UH}_\CC
}$$
and the corollary follows.
\end{proof}

The previous results have the following consequence:

\begin{theorem} \label{hom(G,UH) weakly homotopy equivalent to point}
Let $G$  be a compact Lie group which moreover is  connected or abelian,  and $\CC\subset \Irrep(G)$ a choice of irreducible
representations. Then the space $\homst{G}{\UH}_\CC$ is weakly homotopy equivalent to a point.
\end{theorem}

\begin{proof}
We know that $\homst{G}{\UH}_\CC$ is connected since the conjugation map
\begin{align*}
\pi_\alpha : \UH  & \to \homst{G}{\UH}_\CC\\
  F  & \mapsto F \alpha F^{-1}\end{align*}
is surjective for any choice of $\alpha \in \homst{G}{\UH}_\CC$.

Take any base point map $$f : (S^n, *) \to ( \homst{G}{\UH}_\CC, \alpha)$$ and note that by
Theorem  \ref{theorem existence of sections abelian compact} and Corollary \ref{corollary existence section compact and connected} there exists $\sigma : (S^n,*) \to (\UH,\Id)$ such that
$f= \pi_\alpha \circ \sigma$. Since $\UH$ is contractible by \cite[\S11, Lem. 3]{DixDou}, there exists $\widetilde{\sigma} :(B^{n+1},*) \to (\UH, \Id)$ such that $\widetilde{\sigma}|_{S^n}= \sigma$. Hence we have that $$\pi_\alpha \circ \widetilde{\sigma}: (B^{n+1},*) \to  (\homst{G}{\UH}_\CC, \alpha)$$ satisfies $\pi_\alpha \circ \widetilde{\sigma}|_{S^n}= f$ and therefore the homotopy groups of $\homst{G}{\UH}_\CC$ 
are all trivial.

\end{proof}

\section{\small Spaces of projective unitary representations} \label{Chapter projective unitary representations}

Let $\HH$ be a separable and infinite dimensional Hilbert space. Let $\UH$ denote the group of unitary operators of $\HH$ and let $\PUH$ be the group of projective unitary operators, defined as the quotient of $\UH$ by its center,
$$PU(\mathcal{H}) = \frac{U(\mathcal{H})}{\{ \zeta \cdot \mathrm{Id}\mid |\zeta| =1\} } = U(\mathcal{H})/S^1 $$
where the center $Z(\UH)$ is identified with $S^1$. Then $\PUH$ fits in the following short exact sequence of groups
$$1 \to S^1 \hookrightarrow \UH \stackrel{\pi}{\to} \PUH \to 1 .$$ 
and moreover in \cite[Thm. 1]{Simms:70} it is shown that the above sequence is a $S^1$-principal bundle; in other words,
the quotient map $\UH \stackrel{\pi}{\to} \PUH$ has local cross sections. 

The group of projective unitary operators may be endowed with the strong operator topology, and
in \cite[Thm. 1]{Simms:70} it is shown that this topology agrees with the quotient
topology. Since $S^1$ acts on $\UH$ by isometries, we may endow the group $\PUH$ with the metric defined by the distance between the orbits, i.e. for
$T, U \in \PUH$ define
\begin{align*}
\langle \langle T , & U \rangle \rangle := \\
&\min \{ \langle \widetilde{T}, \widetilde{U} \rangle \colon
 \widetilde{T} \in \pi^{-1}(T), \widetilde{U} \in \pi^{-1}(U) \}.\end{align*}

Let $G$ be a compact Lie group and let $$\alpha : G \to \PUH$$ be a continuous homomorphism.
The homomorphism $\alpha$ defines a projective and unitary representation of $G$ on
the projective Hilbert space $\mathbb P\HH
:=\HH -\{0\}/{\mathbb C}^*$. 

Every homomorphism $\alpha : G \to \PUH$ defines a 
group $\widetilde{G}_\alpha:= \alpha^* \UH$ as the pullback of $\UH$ under $\alpha$. 
The group $\widetilde{G}_\alpha$ is a $S^1$-central extension of the group $G$ and
fits into the diagram
\begin{align*}
\xymatrix{ S^1 \ar[d] \ar[r]^\cong & S^1 \ar[d] \\
\widetilde{G}_\alpha \ar[d]^p \ar[r]^{\widetilde{\alpha}} & \UH \ar[d]^\pi \\
G \ar[r]^\alpha & \PUH
}
\end{align*}
where the bottom square is a pullback square, $\widetilde{\alpha}: \widetilde{G}_\alpha \to \UH$
is the induced continuous homomorphism and $p: \widetilde{G}_\alpha \to G$ is the projection
homomorphism.

Since the kernel of the homomorphism $$p: \widetilde{G}_\alpha \to G$$ acts
 on $\HH$ by multiplication, we only need to consider  irreducible representations of 
the group $\widetilde{G}_\alpha$ where the kernel of $p$ acts by multiplication. Consider
the set
\begin{align*}
\SSS(\widetilde{G}_\alpha) := \{ V \in \Irrep(\widetilde{G}_\alpha) \mid 
 \ker(p: \widetilde{G}_\alpha \to G) & \\ \ \mbox{acts by multiplication on} \ & V \}\end{align*}
and make the following definition:

\begin{definition} \label{definition stable homomorphism in PUH}
A continuous homomorphism $\alpha : G \to \PUH$ will be called {\it stable}
whenever the induced homomorphism $\widetilde{\alpha} : \widetilde{G}_\alpha \to \UH$
belongs to $\homst{\widetilde{G}_\alpha}{\UH}_{\SSS(\widetilde{G}_\alpha)}$.
Denote the set of stable homomorphisms from $G$ to $\PUH$ by $\homst{G}{\PUH}$.
\end{definition}

Since $G$ is compact and $\PUH$ is a metric space, the set of stable homomorphisms
$$\homst{G}{\PUH}$$ may be endowed with the supremum metric, i.e. for $\alpha, \beta \in
\homst{G}{\PUH}$ let
$$\langle \langle \alpha, \beta \rangle \rangle := \sup \{ \langle \langle \alpha(g), \beta(g) \rangle \rangle \mid g \in G \}.$$
By \cite[Thm. 46.8]{Munkres:Top} this metric induces the compact-open topology on $\homst{G}{\PUH}$.

Recall that a $S^1$-central extension of a Lie group $G$ is an extension $\widetilde{G}$
of $G$ which fits in the short exact sequence of Lie groups
$$1 \to S^1 \to \widetilde{G} \stackrel{p}{\to} G \to 1$$
and such that $S^1$ is a subgroup of the center $Z(\widetilde{G})$. Since Lie groups
are locally compact and $S^1$ is compact, the projection map $\widetilde{G} \stackrel{p}{\to} G$
is moreover a $S^1$-principal bundle.

Two $S^1$-central extensions $$\widetilde{G}_0 \stackrel{p_0}{\to} G \ \ \mbox{and} \ \ \widetilde{G}_1 \stackrel{p_1}{\to} G $$ of $G$ are isomorphic
as $S^1$-central extensions, if there exist an isomorphism of groups $$\phi: \widetilde{G}_0 \stackrel{\cong}{ \to} \widetilde{G}_1$$  such that  $p_1 \circ \phi = p_0$.
Denote by $\Ext(G,S^1)$ the set of isomorphism classes of $S^1$-central extensions of $G$
and denote by $[\widetilde{G}]$ the isomorphism class of an extension.

\begin{proposition} \label{proposition S1 central extensions}
Let $G$ be a compact Lie group. Then the canonical map $$\homst{G}{\PUH} \to \Ext(G,S^1), \ \ \alpha \mapsto [\widetilde{G}_\alpha]$$ induces an isomorphism of sets at the level of the conjugacy classes of stable homomorphisms, i.e.
$$\homst{G}{\PUH}/\PUH \cong \Ext(G,S^1).$$
\end{proposition}

\begin{proof}
Let us show first that every isomorphism class of a  $S^1$-central extension arises as the
pullback of a stable homomorphism from $G$ to $\PUH$. Consider an extension $\widetilde{G} \to G$, take the Hilbert space $\HH':=L^2(\widetilde{G}) \otimes L^2([0,1])$ where $L^2$ denotes
square integrable functions and take the standard action of $\widetilde{G}$  on $L^2(\widetilde{G})$. By Peter-Weyl's Theorem the Hilbert space $\HH'$ contains
each irreducible representations of $\widetilde{G}$ infinitely number of times.  Take
the isotypical part corresponding to the irreducible representations in $\SSS(\widetilde{G})$
$$\HH:= \bigoplus_{V \in \SSS(\widetilde{G})} \HH'^V$$
and note that the induced action of $\widetilde{G}$ on $\HH$ is unitary and that
the kernel of the projection $\widetilde{G} \to G$ acts on $\HH$ by multiplication. Therefore
the action $\widetilde{G} \to \UH$ fits into the diagram
\begin{align*}
\xymatrix{ S^1 \ar[d] \ar[r]^\cong & S^1 \ar[d] \\
\widetilde{G} \ar[r]& \UH 
}
\end{align*}
thus inducing a homomorphism $G \to \PUH$. Therefore all $S^1$-central extensions
of $G$ appear as pullbacks of stable homomorphisms from $G$ to $\PUH$.

Now let us consider two homomorphisms $\alpha, \beta  \in \homst{G}{\PUH})$
which are conjugate, i.e. there exist $F \in \PUH$ such that $F\alpha(g)F^{-1}= \beta(g)$ for 
all $g \in G$. Hence the groups $\widetilde{G}_\alpha := \{(g,T) \in G \times \UH \mid
\alpha(g)=\pi(T) \}$ and $\widetilde{G}_\beta := \{(g,T) \in G \times \UH \mid
\beta(g)=\pi(T) \}$ are isomorphic as $S^1$-central extensions; the map
$\widetilde{G}_\alpha \to \widetilde{G}_\beta$, $(g,T) \mapsto (g, FTF^{-1})$
is the desired isomorphism. Denoting by $\homst{G}{\PUH})/\PUH$ the set of conjugacy classes
of stable homomorphisms we obtain a surjective map
$$\homst{G}{\PUH}/\PUH \twoheadrightarrow \Ext(G,S^1).$$

Let us now suppose that the groups $\widetilde{G}_\alpha$ and $\widetilde{G}_\beta$
are isomorphic as $S^1$-central extensions of $G$. The homomorphisms $\widetilde{\alpha}
: \widetilde{G}_\alpha \to \UH$ and $\widetilde{\beta}
: \widetilde{G}_\beta \to \UH$ induce decompositions of $\HH$ by isotypical components
$\HH \cong \bigoplus_{V \in \SSS( \widetilde{G}_\alpha) } \HH^V_\alpha$ and 
$\HH \cong \bigoplus_{W \in \SSS( \widetilde{G}_\beta) } \HH^W_\beta$. The isomorphism
$\phi : \widetilde{G}_\alpha \stackrel{\cong}{\to} \widetilde{G}_\beta$ induces a canonical isomorphism
between the sets $\SSS( \widetilde{G}_\alpha)$ and $\SSS( \widetilde{G}_\beta)$, and 
therefore there exist a unitary isomorphism between $\bigoplus_{V \in \SSS( \widetilde{G}_\alpha) } \HH^V_\alpha$ and 
$ \bigoplus_{W \in \SSS( \widetilde{G}_\beta) } \HH^W_\beta$
which is compatible with the actions of the groups and the isomorphism $\phi$. Therefore
there exists a unitary isomorphism $F \in \UH$ such that $F\widetilde{\alpha}F^{-1}= \widetilde{\beta}$ and this implies that $\alpha$ and $\beta$ are conjugate. Hence
we have an isomorphism of sets $\homst{G}{\PUH}/\PUH \stackrel{\cong}{\to} \Ext(G,S^1).$
\end{proof}

Take representatives $\widetilde{G}$ for each isomorphism class of $S^1$-central extension and denote by
\begin{align*}
\homst{G}{\PUH}_{\widetilde{G}}:=\{\alpha \in \homst{G}{\PUH} \mid & \\
\widetilde{G}_\alpha \cong \widetilde{G} \ \mbox{as} \ S^1  \mbox{-central extensions of} \ G & \}\end{align*}
the space of stable homomorphisms from $G$ to $\PUH$ which define a $S^1$-central extension
isomorphic to $\widetilde{G}$. In view of Proposition \ref{proposition S1 central extensions} 
we have that
\begin{align*}
\homst{G}{\PUH} = & \\
 \bigsqcup_{[\widetilde{G}] \in  \Ext(G,S^1)} &\homst{G}{\PUH}_{\widetilde{G}}.
 \end{align*}

Consider $\widetilde{\alpha} \in  \homst{\widetilde{G}}{\UH}_{\SSS(\widetilde{G})}$
and note that by definition of the set of irreducible representations $\SSS(\widetilde{G})$,
the homomorphism $\widetilde{\alpha}$ makes the following diagram of homomorphisms commutative
$$\xymatrix{
S^1 \ar@{^(->}[d] \ar[r]^\cong & S^1 \ar@{^(->}[d] \\
\widetilde{G} \ar[r]^{\widetilde{\alpha}} & \UH.
}$$

Therefore ${\widetilde{\alpha}} $ induces a continuous homomorphism $\Psi(\widetilde{\alpha}) \in \homst{G}{\PUH}$, and we may define a map of sets
\begin{align*}
\Psi: \homst{\widetilde{G}}{\UH}_{\SSS(\widetilde{G})} &\to \homst{G}{\PUH}_{\widetilde{G}} \\
 \widetilde{\alpha}  & \mapsto \Psi(\widetilde{\alpha}).\end{align*}

Now consider the abelian group $\hom(G,S^1)$ of continuous homomorphisms from
$G$ to $S^1$ endowed with the group structure given by pointwise multiplication.
For every $$\widetilde{\alpha} \in \homst{\widetilde{G}}{\UH}_{\SSS(\widetilde{G})}$$ 
and $\eta \in \hom(G,S^1)$ define $\eta \cdot \widetilde{\alpha}: \widetilde{G} \to \UH$
by 
$$\eta \cdot \widetilde{\alpha} (\widetilde{g}) :=\eta(\pi(\widetilde{g})) \widetilde{\alpha} (\widetilde{g}).$$
The homomorphism $\eta \cdot \widetilde{\alpha}$ belongs to $$\homst{\widetilde{G}}{\UH}_{\SSS(\widetilde{G})}$$ since the action of $\ker(\widetilde{G} \to G)$
is unaffected and therefore the representations that $\eta \cdot \widetilde{\alpha}$ define
belong to $\SSS(\widetilde{G})$.

Therefore we have an action  of $\hom(G,S^1)$ on $\homst{\widetilde{G}}{\UH}_{\SSS(\widetilde{G})}$ as follows:
\begin{align*}
\hom(G,S^1) \times \homst{\widetilde{G}}{\UH}_{\SSS(\widetilde{G})} &  \to \\
\mathrm{hom}_{\rm{st}}&(\widetilde{G},\UH)_{\SSS(\widetilde{G})}\\ \ (\eta, \widetilde{\alpha})& \mapsto
\eta \cdot \widetilde{\alpha}.\end{align*}

We claim the following theorem.

\begin{theorem} \label{theorem principal bundle}
Let $G$ be a compact Lie group which is connected or abelian, and let $\widetilde{G}$ be a $S^1$-central extension of $G$. 
Let $\SSS(\widetilde{G})$ be the set of isomorphism classes of irreducible representations
of $\widetilde{G}$ on which $\ker(\widetilde{G} \to G)$ acts by multiplication of scalars.
Then the map
\begin{align*}
 \Psi: \homst{\widetilde{G}}{\UH}_{\SSS(\widetilde{G})} & \to \homst{G}{\PUH}_{\widetilde{G}}\\
\widetilde{\alpha} & \mapsto \Psi(\widetilde{\alpha})\end{align*}
is a $\hom(G,S^1)$-principal bundle, and in particular a local homeomorphism. 
\end{theorem}

\begin{proof}
We already know that $\Psi$ is surjective; the space $\homst{G}{\PUH}_{\widetilde{G}}$
consists of the homomorphisms $\alpha$ such that $\widetilde{G}_\alpha \cong \widetilde{G}$
as $S^1$-principal bundles. The continuity of $\Psi$ follows from the inequality
$\langle \widetilde{\alpha}, \widetilde{\beta} \rangle \geq \langle \langle \Psi(\widetilde{\alpha}), \Psi(\widetilde{\beta}) \rangle \rangle$ since by definition
\begin{align*}
 \langle \langle \Psi(\widetilde{\alpha}), \Psi(\widetilde{\beta}) \rangle \rangle
=  &\sup_{g \in G}  \langle \langle \Psi(\widetilde{\alpha})(g), \Psi(\widetilde{\beta})(g) \rangle \rangle\\
 \leq &\sup_{\widetilde{g} \in \widetilde{G}}  \langle \widetilde{\alpha}(\widetilde{g}), \widetilde{\beta}(\widetilde{g}) \rangle =\langle \widetilde{\alpha}, \widetilde{\beta} \rangle.
\end{align*}

Take now $\widetilde{\alpha}, \widetilde{\alpha}'$ such that $\Psi(\widetilde{\alpha})=
\Psi(\widetilde{\alpha}')$. Note that for all $\widetilde{g} \in \widetilde{G}$
the product $\widetilde{\alpha}(\widetilde{g})\widetilde{\alpha}'(\widetilde{g})^{-1}$
belongs to $S^1 \subset \UH$ and therefore we can define the assignment
$$\widetilde{\eta}: \widetilde{G} \to S^1, \ \ \widetilde{g} \mapsto \widetilde{\alpha}(\widetilde{g})\widetilde{\alpha}'(\widetilde{g})^{-1}.$$
This assignment is indeed a homomorphism since we have the equalities
\begin{align*}
\widetilde{\eta}(\widetilde{g}\widetilde{h}) =& \widetilde{\alpha}(\widetilde{g}\widetilde{h})\widetilde{\alpha}'(\widetilde{g}\widetilde{h})^{-1}\\=
 &\widetilde{\alpha}(\widetilde{g})\widetilde{\alpha}(\widetilde{h})\widetilde{\alpha}'(\widetilde{h})^{-1}\widetilde{\alpha}'(\widetilde{g})^{-1}\\
= & \widetilde{\alpha}(\widetilde{g})\widetilde{\eta}(\widetilde{h})\widetilde{\alpha}'(\widetilde{g})^{-1}\\= &
\widetilde{\eta}(\widetilde{g})\widetilde{\eta}(\widetilde{h})
\end{align*}
which follow from the fact that $\widetilde{\eta}(\widetilde{h})$ lies on the center of $\UH$.
The homomorphism $\widetilde{\eta}$  is trivial once restricted to $\ker(\widetilde{G} \to G)$ and therefore it induces a homomorphism $\eta: G \to S^1$ such that $\widetilde{\eta}(\widetilde{g})=\eta(\pi(\widetilde{g}))$. Therefore we obtain the equation $\eta \cdot \widetilde{\alpha}'=\widetilde{\alpha}$ which implies that $\Psi$ induces a
bijective map at the level of sets
\begin{align} 
\nonumber \homst{\widetilde{G}}{\UH}_{\SSS(\widetilde{G})} / \hom(G,&S^{1})  \stackrel{\cong}{\to}\\ 
\label{bijective map of sets between G tilde and G}&\homst{G}{\PUH}_{\widetilde{G}}.\end{align}

We need to show now that $\Psi$ is a local homeomorphism.
Note that for any non trivial $\eta \in \hom(G,S^1) $ and $\mathbf{1}$
the trivial homomorphism, we have that $\langle \mathbf{1},\eta \rangle \geq \sqrt{2}$  since any non trivial homomorphism
must take at least one value in the subset $\{e^{i t} \mid \frac{\pi}{2}< t< \frac{3\pi}{2} \} \subset S^1.$ 
 This implies that for any $\widetilde{\alpha} \in \homst{\widetilde{G}}{\UH}_{\SSS(\widetilde{G})}$ and any $\eta \in \hom(G,S^1) $ we have
 $$\langle \widetilde{\alpha}, \eta \cdot \widetilde{\alpha} \rangle = 
 \langle {\mathbf{1}},\eta \rangle \geq \sqrt{2};$$ therefore if we denote
 by $$B_{\delta}(\widetilde{\alpha}) := \{ \widetilde{\beta} \in \homst{\widetilde{G}}{\UH}_{\SSS(\widetilde{G})}  \mid
 \langle \widetilde{\alpha}, \widetilde{\beta} \rangle < \delta \}$$
we have that for all $\eta \in \hom(G, S^1)$ with $\eta \neq \mathbf{1}$, 
the intersection $B_{\delta}(\widetilde{\alpha}) \bigcap \left(\eta \cdot
B_{\delta}(\widetilde{\alpha})\right) = \varnothing$ for $\delta < \frac{1}{2}$, which in particular
says that the action of $\hom(G,S^1)$ is completely discontinuous.

Let us restrict the map $\Psi$ to the open set $B_{\delta}(\widetilde{\alpha})$ with $\delta \ll \frac{1}{2}$. By equation \eqref{bijective map of sets between G tilde and G}
we have that the map
$$\Psi|_{B_{\delta}(\widetilde{\alpha})} \colon B_{\delta}(\widetilde{\alpha}) \to \Psi(B_{\delta}(\widetilde{\alpha}))$$
is bijective and continuous, we claim furthermore that it is a homeomorphism. Let us show that
${\Psi|_{B_{\delta}(\widetilde{\alpha})}}^{-1}$ is continuous.

Let $\widetilde{K}$ be a maximal abelian subgroup of $\widetilde{G}$ and denote
by $\widetilde{\alpha}': \widetilde{K} \to \UH$ the restriction $\widetilde{\alpha}|_{\widetilde{K}}$.
Denote by $K \subset G$ the abelian subgroup that $\widetilde{K}$ defines and note that
 $\widetilde{K} \stackrel{p}{\to} K$ defines a $S^1$-central extension.
 Denote by $$\HH \cong \bigoplus_{W \in \SSS(\widetilde{K})} \HH^W_{\widetilde{\alpha}'}$$ the decomposition of $\HH$ into isotypical
 components and note that $S^1$ acts by multiplication on the one dimensional irreducible
 representations $W$ of $\SSS(\widetilde{K})$.
 Fix $V \in \SSS(\widetilde{K}) $ and choose any unitary vector $x \in \HH^V_{\widetilde{\alpha}'}$.
 
 Take $\varepsilon \in \mathbb R $ such that $0 < \varepsilon \ll 1$ and choose $\delta < \frac{1}{2}$ such that for any
  $\widetilde{\beta} \in B_{\delta}(\widetilde{\alpha})$ we have 
 that for all $\widetilde{g} \in \widetilde{G}$
\begin{align}\label{inequality epsilon} \left|\widetilde{\alpha}(\widetilde{g})x - \widetilde{\beta}(\widetilde{g})x \right|< \varepsilon;
\end{align}
this $\delta$ exists by the definition of the strong operator topology and the metric defined in
equation \eqref{metric UH}.

Take a sequence $\{\widetilde{\beta}_n\}_{n \in \mathbb N}$ of homomorphisms in 
$B_{\delta}(\widetilde{\alpha})$ and denote 
by $$\alpha, \beta_n \in \homst{G}{\PUH}_{\widetilde{G}}$$ the projective homomorphisms that $\widetilde{\alpha}$ and
$\widetilde{\beta}_n$ define. Assume that $\lim_{n \to \infty} \beta_n = \alpha$; let us show
that this implies that $\lim_{n \to \infty} \widetilde{\beta}_n = \widetilde{\alpha}$.

Since we have that $\lim_{n \to \infty} \beta_n = \alpha$, there must exist unitary
complex numbers  $\lambda_n(g) \in S^1$ such that for all $\widetilde{g} \in \widetilde{G}$
$$\lim_{n \to \infty} \lambda_n(g)\widetilde{\beta}_n(\widetilde{g}) = \widetilde{\alpha}(\widetilde{g}).$$

Take $\sigma \in \mathbb R$ such that $0 < \sigma \ll1$ and let $N \in \mathbb N$ be such that
for all $n > N$  and all $\widetilde{g} \in \widetilde{G}$ we have 
\begin{align} \label{inequality sigma}
| \lambda_n(g)\widetilde{\beta}_n(\widetilde{g})x - \widetilde{\alpha}(\widetilde{g})x| < \sigma. 
\end{align}

Denote by $\chi_W: \widetilde{K} \to S^1$ the characters of the irreducible representations $W$
of the abelian group $\widetilde{K}$, and write $x = \sum_{W \in \SSS(\widetilde{K})} y_n^W$
in terms of the decomposition on isotypical components $\HH \cong \bigoplus_{W \in \SSS(\widetilde{K})} \HH^W_{\widetilde{\beta}_n'}$ relative to $\widetilde{\beta}_n'$ with $y_n^W \in \HH^W_{\widetilde{\beta}_n'}$.
We obtain that for all $\widetilde{k} \in \widetilde{K}$ 
$$\widetilde{\alpha}(\widetilde{k})x = \chi_V(\widetilde{k}) x = \sum_W \chi_V(\widetilde{k}) y_n^W$$
and
$$ \widetilde{\beta}_n(\widetilde{k})x = \sum_W \chi_W(\widetilde{k}) y_n^W,$$
and therefore by equation \eqref{inequality epsilon}
\begin{align*} \left| \widetilde{\beta}_n(\widetilde{k})x-\widetilde{\alpha}(\widetilde{k})x \right|
=
\left| \sum_{W \neq V} (\chi_W(\widetilde{k})  - \chi_V(\widetilde{k}))y_n^W \right| < \varepsilon\end{align*}
and by equation \eqref{inequality sigma}
\begin{align}  \nonumber 
\left| \lambda_n(k)\widetilde{\beta}_n(\widetilde{k})x - \widetilde{\alpha}(\widetilde{k})x \right|
&=\\ \nonumber 
\left|\chi_V(\widetilde{k})(\lambda_n(k)-1)y_n^V+\right. & \left. \sum_{W \neq V} (\chi_W(\widetilde{k})  - \chi_V(\widetilde{k}))y_n^W \right|\\ 
&< \sigma.
\label{inequality sigma decomposed}
\end{align}

Since for all $\widetilde{k} \in \widetilde{K}$ we have that $$ \left| \sum_{W \neq V} (\chi_W(\widetilde{k})  - \chi_V(\widetilde{k}))y_n^W \right| < \varepsilon,$$ Lemma \ref{lemma distance between reps implies distace}
shows that  $\left|\sum_{W \neq V} y_n^W \right| < \varepsilon$. Therefore we obtain
 $$ \left|y_n^V \right|= \left|x - \sum_{W \neq V} y_n^W \right| > 1- \varepsilon.$$

Since the vectors $y_n^W$ are pairwise orthogonal, the inequality \eqref{inequality sigma decomposed} implies that
$$\left|\chi_V(\widetilde{k})(\lambda_n(k)-1)y_n^V\right| = \left|\lambda_n(k)-1 \right| \, \left|y_n^V \right| < \sigma $$
and since $\left|y_n^V \right|> 1- \varepsilon$, we have that for all $k \in K$ and all $n >N$
$$ \left|\lambda_n(k)-1 \right| < \frac{\sigma}{1- \varepsilon}.$$

Since $\varepsilon$ is fixed, we conclude  that for all $k \in K$, $\lim_{n \to \infty} \lambda_n(k) =1$. Hence for
all $\widetilde{k} \in \widetilde{K}$
$$\lim_{n \to \infty} \widetilde{\beta}_n(\widetilde{k}) =
\lim_{n \to \infty} \lambda_n(k)\widetilde{\beta}_n(\widetilde{k}) = \widetilde{\alpha}(\widetilde{k})$$
and therefore $\lim_{n \to \infty}  \widetilde{\beta}_n' =  \widetilde{\alpha}'$. Now, the restriction map 
\begin{align*}
\homst{\widetilde{G}}{\UH}_{\SSS(\widetilde{G})} & \to \homst{\widetilde{K}}{\UH}_{\SSS(\widetilde{K})}\\
  \widetilde{\alpha} &\mapsto \widetilde{\alpha}'=\widetilde{\alpha}|_{\widetilde{K}}\end{align*}
is an embedding since $\widetilde{G}$ is connected and any representation of $\widetilde{G}$
is uniquely determined by its restriction to $\widetilde{K}$; hence we conclude that
$$\lim_{n \to \infty}  \widetilde{\beta}_n =  \widetilde{\alpha}.$$ Therefore 
$\Psi|_{B_{\delta}(\widetilde{\alpha})} \colon B_{\delta}(\widetilde{\alpha}) \to \Psi(B_{\delta}(\widetilde{\alpha}))$ is a homeomorphism and the theorem follows.
\end{proof}

\begin{corollary}
Let $G$ be a compact Lie group which is connected or abelian, and let $\widetilde{G}$ be a $S^1$-central extension of $G$. Then $\homst{G}{\PUH}_{\widetilde{G}}$ is a $K(\hom(G,S^1),1)$, namely
it is connected, its fundamental group is $\hom(G,S^1)$ and its higher homotopy groups are trivial.
\end{corollary}
\begin{proof} The result follows from Theorem \ref{theorem principal bundle} where it is proven that
$$\homst{\widetilde{G}}{\UH}_{\SSS(\widetilde{G})} \to \homst{G}{\PUH}_{\widetilde{G}}$$
is a $\hom(G, S^1)$-principal bundle, and from Theorem \ref{hom(G,UH) weakly homotopy equivalent to point} where it is proven that 
$$\homst{\widetilde{G}}{\UH}_{\SSS(\widetilde{G})}$$
is weakly homotopy equivalent to a point.
\end{proof}

\begin{lemma}\label{lemma distance between reps implies distace} Let $\widetilde{K}$ be an abelian
compact Lie group and denote by $\chi_W: \widetilde{K} \to S^1$ the character
of the 1-dimensional irreducible representation $W$. For each irreducible representation $W$ 
take a vector $y^W \in \HH$ and assume that $y^W \perp y^Z$ for $W \neq Z$. Suppose that for all $\widetilde{k} \in \widetilde{K}$
$$ \left| \sum_{W \neq V} (\chi_W(\widetilde{k})  - \chi_V(\widetilde{k}))y_n^W \right| < \varepsilon,$$
then $\left|\sum_{W \neq V} y_n^W \right| < \varepsilon$.
\end{lemma}
\begin{proof}
Note first that $\chi_W (\chi_V)^{-1}= \chi_{W \otimes V^{-1}}$ and therefore
\begin{align*} \left|\sum_{W \neq V} (\chi_W(\widetilde{k})  - \chi_V(\widetilde{k}))y^W \right|& =\\ \left|\sum_{Z \neq \mathbf{1}} (\chi_Z(\widetilde{k}) \right. & \left. - 1)y^{V \otimes Z} \right|;\end{align*}
taking $x^Z:=y ^{V \otimes Z}$ we have that for all $\widetilde{k} \in \widetilde{K}$
$$ \left|\sum_{Z \neq \mathbf{1}} (\chi_Z(\widetilde{k})  - 1)x^{ Z}\right| < \varepsilon.$$
Since $\widetilde{K}$ is isomorphic to a product of cyclic groups, we claim that it is enough to
show the lemma whenever $\widetilde{K}$ is $S^1$.

Here the irreducible representations of $S^1$ are parametrized by $n \in \mathbb Z$ and our hypothesis
becomes
$$ \left| \sum_{n \neq 0}( e^{2 \pi i n j}-1)x_n \right| < \varepsilon$$
for all $j \in \mathbb R$. Take a prime number $ p \in \mathbb N$ and consider 
the inequality for $j= \frac{1}{2p}, \frac{2}{2p},...,\frac{2p-1}{2p}$; by the triangle inequality we have
\begin{align*}
 \left| \sum_{r=1}^{2p-1} \sum_{n \neq 0}( e^{2 \pi i n \frac{r}{2p}}-1)x_n \right|  \leq & 
 \sum_{r=1}^{2p-1} \left| \sum_{n \neq 0}( e^{2 \pi i n \frac{r}{2p}}-1)x_n \right| \\
  < & (2p-1) \varepsilon 
\end{align*}
which we may reorder thus obtaining
\begin{align*}
(2p-1) \varepsilon 
> & \left| \sum_{r=1}^{2p-1} \sum_{n \neq 0}( e^{2 \pi i n \frac{r}{2p}}-1)x_n \right| \\
= &	\left| \left( \sum_{n \colon p \nmid n} x_n\right)\left(\sum_{r=1}^{2p-1} ( e^{2 \pi i n \frac{r}{2p}}-1)\right)\right| \\ 
= & 2p\left| \sum_{n \colon p \nmid n} x_n \right|
\end{align*}
since $\sum_{r=1}^{2p-1} ( e^{2 \pi i n \frac{r}{2p}}-1)=-2p$ whenever $p \nmid n$.

Therefore we have that for all prime $p$ we have that the inequalities $$\left| \sum_{n \colon p \nmid n} x_n \right| < \frac{(2p-1) \varepsilon}{2p} < \varepsilon$$ hold, implying the desired result, namely that $$ \left| \sum_{n \neq 0} x_n \right| < \varepsilon.$$

The iteration of the previous argument shows the lemma for any abelian compact Lie group.
\end{proof}

\begin{theorem} \label{theorem existence local section compact and connected for PUH}
 Let $G$ be a compact Lie group which is connected or abelian, and let $\widetilde{G}$ be a $S^1$-central extension of $G$, $B$ a connected paracompact
space of finite covering dimension and $f: B \to \homst{G}{\PUH}_{\widetilde{G}}$ a continuous map.
 Then for all $b_0 \in B$ there exist a neighborhood $V \subset B$ of $b_0 \in V$ such
that $f|_V : V \to \homst{G}{\PUH}_{\widetilde{G}}$ has an extension
$\sigma_V: V \to \PUH$ that makes the following diagram commutative
$$\xymatrix{
&& \PUH \ar[d]^-{\bar{\pi}_{f(b_0)}} \\
V \ar[rr]^-{f|_V} \ar[rru]^-{\sigma_V}& &\homst{G}{\PUH}_{\widetilde{G}}
}$$
where $\bar{\pi}_{f(b_0)}: \PUH \to \homst{G}{\PUH}_{\widetilde{G}}$, $F \mapsto F f(b_0)F^{-1}$.
In particular $f^* \PUH$ is a $\PUH_{f(b_0)}$-principal bundle where
 $\PUH_{f(b_0)}:=\{F \in \PUH \mid Ff(b_0)F^{-1}=f(b_0)\}$.
\end{theorem}
\begin{proof}
Take any lift $\widetilde{\alpha} \in \homst{\widetilde{G}}{\UH}_{\SSS(\widetilde{G})}$ such that
$\Psi(\widetilde{\alpha}) = f(b_0)$ and consider the following commutative diagram
$$
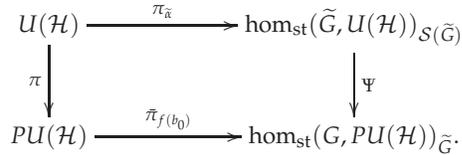
\xymatrix{
\UH \ar[rr]^-{\pi_{\widetilde{\alpha}}} \ar[d]_-{\pi} &&  \homst{\widetilde{G}}{\UH}_{\SSS(\widetilde{G})} \ar[d]^-{\Psi}\\
 \PUH \ar[rr]^-{\bar{\pi}_{f(b_0)}} & & \homst{G}{\PUH}_{\widetilde{G}}.
}$$
By Theorem \ref{theorem principal bundle} there exists a neighborhood $W $ of $\widetilde{\alpha}$
such that $\Psi|_W : W \to \Psi(W)$ is a homeomorphism. Let $V := f^{-1}(\Psi(W))$ and consider the map
$\widetilde{f} : V \to W$, $ b \mapsto (\Psi|_V)^{-1} (f(b))$. By  Theorem \ref{theorem existence of sections abelian compact} and Corollary
 \ref{corollary existence section compact and connected} there exist a section $s: V \to \UH$ such that
 $s(b) \widetilde{\alpha} s(b)^{-1} = \widetilde{f}(b)$. The composition
 $\sigma_V : V \to \PUH$, $\sigma_V:= \pi \circ s$ is desired local extension of $f$ since we have
 \begin{align*}
 \sigma_V(b) f(b_0) \sigma_V(b)^{-1} =& \pi \left( s(b) \widetilde{\alpha} s(b)^{-1} \right)  \\
= & \pi (\widetilde{f}(b) ) = f(b).\end{align*} 

The map $V \times \PUH_{f(b_0)} \to (f^*\PUH)|_V$, $(b,F) \mapsto (b,\sigma_V(b)F)$ is the desired
local trivialization for the $\PUH_{f(b_0)}$-principal bundle 
\begin{align*}f^*\PUH := &\\ \{(b,F)   & \in B \times \PUH \mid
f(b)= \bar{\pi}_{f(b_0)}(F) \}.\end{align*}
\end{proof}

\section{\small Applications and further research} \label{Chapter applications}

 One very important application of the existence of lifts
 for paracompact spaces shown in Theorem \ref{theorem existence local section compact and connected for PUH},
 is the construction of universal equivariant
 projective unitary and stable bundles necessary
 for the study of the twisted equivariant K-theory as
 an equivariant parametrized cohomology theory (see \cite[Chapter 5]{BEJU} and \cite[Chapter 6]{Atiyah-Segal:TKT}). The construction of these universal
 bundles relies on the construction of classifying spaces
 for certain families for subgroups, together with explicit
 topological properties that the groups 
 and the associated spaces of homomorphisms need to satisfy.
 In what follows we will review these constructions and 
 we will show the implications that the results proved in the
 previous sections have on the existence of universal
 equivariant projective unitary bundles.

   Let $G$ and $P$ be topological groups. A $G$-equivariant $P$-principal
bundle consists of a $P$-principal bundle $p:E \to X$   together with left $G$ actions on $E$ and $X$ commuting with the right $P$ action on $E$
such that $p$ is $G$-equivariant. For every $e \in E$ we obtain a  local representation $\rho_e: G_{p(e)} \to P$ 
determined by $g^{-1} \cdot e=e \cdot \rho_e(g)$ for $g \in G_{p(e)}$ where $G_{p(e)}$ is the isotropy group 
of $p(e) \in X$.

In \cite[Thm. 11.4]{LueckUribe} it was constructed a universal $G$-equivariant $P$-principal
bundle with a prescribed family of local representations through the use of classifying spaces of
families of subgroups. The fact that these classifying spaces of families of subgroups permitted to obtain
equivariant principal bundles relied on topological properties of the groups $G$ and $P$ and
 on the spaces of prescribed homomorphisms. Let us recall the main ingredients.

A family $\mathcal R$ of local representations for $(G,P)$  is a set of pairs
$(H, \alpha)$, where $H$ is a subgroup of $G$ and $\alpha: H \to P$ is a continuous group homomorphism,
such that the family is closed under finite intersections, under conjugation in $P$ and under conjugation in $G$
(see \cite[Def. 3.3]{LueckUribe} for a detailed description).

It is said \cite[Def. 6.1]{LueckUribe} that the family $\mathcal R$ satisfies Condition (H) if the following holds for every 
$(H, \alpha) \in \mathcal R$:
\begin{itemize}
\item The path component of $\alpha$  in $\hom(H,P)$ is contained in the orbit $\{p \alpha p^{-1} \mid p \in P\}$.
\item The projection $P \to P/P_\alpha$ has a local cross section where $P_{\alpha} = \{p \in P \mid p \alpha p^{-1}= \alpha\}$
is the isotropy group of $\alpha$ under the conjugation action of $P$.
\item The projection $G \to G/H$ has a local cross section.
\item The canonical map \begin{align*}\iota_\alpha : P / P_{\alpha} & \to \hom(H,P) \\ p P_\alpha  & \mapsto p \alpha p^{-1}\end{align*}
 is a homeomorphism into its image.
\end{itemize}

To a family of local representations $\mathcal R$ we can associate a family of subgroups of $G \times P$ consisting of
the set ${\mathcal F}({\mathcal R}): = \{K(H,\alpha) \mid (H, \alpha ) \in {\mathcal R} \}$ where $K(H,\alpha):=\{(g, \alpha(g)) \mid g \in H \}$.  Let $$E(G,P, \mathcal R ):=E_{{\mathcal F}({\mathcal R})}(G \times P)$$ be the classifying space for
the family of subgrups ${\mathcal F}({\mathcal R})$, i.e. a $(G \times P)$-CW-complex whose isotropy groups belong
to ${\mathcal F}({\mathcal R})$ and for which the $K(H, \alpha)$-fixed point set
 $E_{{\mathcal F}({\mathcal R})}(G \times P)^{K(H, \alpha)}$ is nonempty and weakly contractible for every $(H, \alpha) \in \mathcal R$.

Theorem 11.4 of \cite{LueckUribe} claims that if the family of local representations $\mathcal R$ satisfies Condition (H), then
$E(G,P, \mathcal R) \to E(G,P, \mathcal R)/P$ is a $G$-equivariant $P$-principal bundle which is moreover universal
for $G$-equivariant $P$-principal bundles whose local representations appear in $\mathcal R$.

In this paper we are interested in $G$ equivariant projective unitary stable bundles, namely 
$G$-equivariant $\PUH$-principal bundles whose local representations $(H, \alpha)$ consist of
stable homomorphisms $\alpha : H \to \PUH$ as were defined in Definition \ref{definition stable homomorphism in PUH}.
 
Whenever $\PUH$ is endowed with the norm topology (let us denote it by $\PUH_n$),
 $G$ is a topological group and $\mathcal S$ consists
of the family of local representations 
$(H, \alpha)$ where $H$ is a finite subgroup of $G$ and $\alpha \in \homst{H}{\PUH_n}$ is a stable homomorphism.
Theorem 15.12 of \cite{LueckUribe} shows that the bundle $$E(G, \PUH_n, \mathcal S) \to E(G,\PUH_n, \mathcal S)/\PUH_n$$
is a universal $G$ equivariant projective unitary stable bundle for almost free $G$-CW-complexes

It would be expected that a similar statement would hold whenever we expand the family of local representations
for pairs $(H, \alpha)$ where $H$ is a compact Lie group and $\alpha$ is a stable homomorphism. Unfortunately this
is not the case for the following reasons: whenever $H$ is a compact Lie group which is not finite, the space
of stable homomorphisms to $\PUH_n$ in the norm topology is empty, i.e. $\homst{H}{\PUH_n} = \varnothing$. 

If we consider the group $\PUH$ endowed with the strong operator topology, as it is done throughout this article, 
the family of local representations $\mathcal S$ consisting of pairs $(H, \alpha)$ with $H$ a compact Lie
group and $\alpha \in \homst{H}{\PUH}$ does not satisfy Condition (H) of  Theorem 11.4 of \cite{LueckUribe}. 
In particular the canonical map $$\PUH / \PUH_\alpha \to \homst{H}{\PUH}$$ is not a homeomorphism
into its image, since Theorem \ref{no sections homst} implies that canonical map $$\PUH \to \homst{H}{\PUH}, \ \ F \to F\alpha F^{-1}$$
has no local sections.

Nevertheless, by Theorem \ref{theorem existence local section compact and connected for PUH} we know that 
 local lifts exists if we restrict to maps $$B \to \homst{H}{\PUH}$$ with $B$  paracompact. Hence we might say that
 a Weak Condition (H) is satisfied whenever Condition (H) holds on the image of maps $B \to \homst{H}{\PUH}$ where
  $B$ is paracompact. With this setup in mind, we conjecture
 that the space $$ E(G, \PUH, \mathcal S)/ \PUH$$ would become a universal space for $G$-equivariant projective unitary stable
 bundles whenever we restrict our study to the category of paracompact spaces with proper $G$ actions.
 If this were the case, we would have a space
 that would allow us to show that the twisted equivariant
 K-theory is indeed an equivariant parametrized cohomology
 theory as defined in \cite{May:EHCT}.
 
   Finally note that in order for the previous statement
 to be true we would need to be able to generalize Theorem \ref{theorem existence local section compact and connected for PUH}
 for compact Lie groups which are not necessarily connected, and we would need to show that the proof
 of Theorem 11.4 of \cite{LueckUribe} would work if we restrict only to the image of paracompact spaces. These tasks are beyond the scope of this article and we leave
them for further research.

\bibliographystyle{chicago}

\renewcommand{\refname}

\end{multicols}
\end{small}
\end{document}